\documentclass[reqno]{amsart}
\usepackage{hyperref}
\usepackage[utf8]{inputenc}
\usepackage[T2A]{fontenc}
\usepackage[russian,english]{babel}
\usepackage[margin=2.54cm,includeheadfoot]{geometry}
\usepackage{amsfonts,amsmath} 
\usepackage{dsfont}
\usepackage{amssymb}
\usepackage{mathrsfs}
\usepackage{graphicx}
\usepackage{tikz}
\usetikzlibrary{intersections,calc}

\date{\today}

\let\d\relax
\newcommand{\d}{\partial}

\newcommand{\R}{\mathbb{R}}
\newcommand{\Q}{\mathbb{Q}}
\newcommand{\N}{\mathbb{N}}
\newcommand{\F}{\mathscr{F}}

\newcommand{\1}{\mathds{1}}
\renewcommand{\L}{\mathscr{L}}
\renewcommand{\H}{\mathscr{H}}
\DeclareMathOperator{\dist}{dist}
\DeclareMathOperator{\diam}{diam}

\newtheorem{thm}{Theorem}[section]

\newtheorem{prop}[thm]{Proposition}
\newtheorem{defn}[thm]{Definition}

\renewcommand{\geq}{\geqslant}
\renewcommand{\leq}{\leqslant}

\usepackage{mathtools}
\mathtoolsset{showonlyrefs}

\begin{document}

\title[]{On the weak Sard property}

\author[R.V. Dribas]{Roman V. Dribas}
\email[R. D.]{dribas.rv@phystech.su}

\author[A.S. Golovnev]{Andrew S. Golovnev}
\email[A. G.]{golovnev.as@phystech.su}

\author[N.A. Gusev]{Nikolay A. Gusev}
\email[N. G.]{ngusev@phystech.su, n.a.gusev@gmail.com}

\address{Moscow Institute of Physics and Technology,
  9 Institutskiy per., Dolgoprudny, Moscow Region, 141700
}

\begin{abstract}

If $f\colon [0,1]^2 \to \R$ is of class $C^2$ then Sard's theorem implies that $f$ has the following \emph{relaxed Sard property}:
the image under $f$ of the Lebesgue measure restricted to the critical set of $f$ is a singular measure.
We show that for $C^{1,\alpha}$ functions with $\alpha<1$ this property is strictly stronger than
the \emph{weak Sard property} introduced by Alberti, Bianchini and Crippa,
while for any monotone continuous function these two properties are equivalent.

We also show that even in the one-dimensional setting H\"older regularity is not sufficient for the relaxed Sard property.
\end{abstract}

\maketitle
%\linenumbers

\section{Introduction}

Let $\alpha \in (0,1)$. 
Let $f\colon [0,1]^d \to \R$ be a continuous function, $d\in \N$.
Let $S$ denote the critical set of $f$, i.e. the set of all points $x\in [0,1]^d$ where $f$ is either non-differentiable or has $\nabla f(x)=0.$

If $f$ was of class $C^{d}$ then by Sard's theorem we would have
\begin{equation}
\L^1(f(S)) = 0,
\label{strong-Sard-property}
\end{equation}
where $\L^d$ denotes the Lebesgue measure on $\R^d$.
Consequently, we would have
\begin{equation}
f_\# (\1_S \L^d) \perp \L^1.
\label{relaxed-Sard-property}
\end{equation}
Here
$\1_S$ is the indicator function of the set $S$,
$f_\# \mu$ denotes the pushforward (image) of the measure $\mu$ by the function $f$,
and $\mu \perp \nu$ means that the measures $\mu$ and $\nu$ are mutually singular.
For convenience we will use the following definition:

\begin{defn}
The function $f$ has the \emph{strong Sard property} if it satisfies
\eqref{strong-Sard-property}.
The function $f$ has the \emph{relaxed Sard property} if it satisfies \eqref{relaxed-Sard-property}.
\end{defn}

For any $t\in \R$ let $E_t:=f^{-1}(t)$ and let $E_t^*$ denote the union of all connected components of $E_t$ with strictly positive length (i.e. Hausdorff measure $\H^1$). Let
$E^*:= \bigcup_{t\in \R} E_t^*$.
Since $f$ is continuous, the set $E^*$ is a Borel set, see \cite[Proposition 6.1]{ABC_2013_LipSard}.

The following definition was introduced in \cite{ABC_2013_LipSard}:

\begin{defn}
A continuous function $f \colon [0,1]^2 \rightarrow \R$ has the \emph{weak Sard property} if
\begin{equation}
f_\# (\1_{S\cap E^*}\L^2) \perp \L^1.
\label{weak-Sard-property}
\end{equation}
\end{defn}

Clearly \eqref{relaxed-Sard-property} implies \eqref{weak-Sard-property}.
If $f$ is \emph{monotone} (i.e. for all $t\in \R$ the level sets $\{f=t\}$ are connected),
% see e.g. \cite[p.~358]{Engelking1989}
then the converse implication holds:

\begin{thm}\label{wsp-eq-rsp-for-monotone-fn}
If $f$ is monotone then the properties \eqref{relaxed-Sard-property} and \eqref{weak-Sard-property} are equivalent.
\end{thm}

(For $f\in W^{1,2}$ similar result was established in \cite[Corollary 5.7]{GK24},
with the monotonicity being understood in the generalized sense introduced by Bianchini and Tonon \cite[Definition 3]{BT}.)

By the well-known result \cite[Theorem 2]{Bates93} if $f\in C^{1,1}$ then $f$ has the strong Sard property
(and consequently both \eqref{relaxed-Sard-property} and \eqref{weak-Sard-property} hold).
However in \cite[Section 5.1]{Bon17a} there was constructed
an example of Lipschitz function $f$ for which the converse implication fails.
Our main result is that \eqref{weak-Sard-property}
does not imply \eqref{relaxed-Sard-property} even if~$f\in C^{1,\alpha}$ with $\alpha\in (0,1)$:

\begin{thm}\label{rsp-vs-wsp}
There exists $f\in \bigcap_{\alpha\in(0,1)} C^{1,\alpha}([0,1]^2)$ which has the weak Sard property
but does not have the relaxed Sard property.
\end{thm}

An example of $C^1$ function without the strong Sard property was constructed in \cite{Whitney1935} (see also \cite{Grinberg1985}).
Similar examples in the classes $C^{1,\alpha}$ with $\alpha \in (0,1)$ were constructed by Norton in \cite[Proposition on page 401]{Norton1989}.
However it is not evident whether the corresponding functions have the relaxed Sard property.

As we already mentioned,
$C^{1,1}$ functions have the strong Sard property. %and, consequently, the relaxed Sard property.
In fact a similar result holds in the one-dimensional case:
if $f\in C^{0, 1}([0,1])$ then $f$ has the strong Sard property.
(This is a simple consequence of Rademacher's theorem and Vitali covering theorem.)
Therefore, one could ask whether in the one-dimensional case certain H\"older regularity of $f$ could be sufficient for the relaxed Sard property.
We show that this is not the case (similar to the two-dimensional setting):

\begin{thm}\label{rsp-1-dim}
    There exist $f \in \bigcap_{\alpha \in (0, 1)} C^{0, \alpha} ([0, 1])$ which does not have the relaxed Sard property.
\end{thm}

A small modification of the proof of Theorem~\ref{rsp-1-dim} also shows that even in the one-dimensional setting the relaxed Sard property is strictly weaker than strong Sard property:

\begin{thm}\label{rsp-1-dim+}
There exist $f \in \bigcap_{\alpha \in (0, 1)} C^{0, \alpha} ([0, 1])$ which does not have the strong Sard property, but has the relaxed Sard property.
\end{thm}

In section~\ref{section-1dim} we prove Theorems~\ref{rsp-1-dim} and~\ref{rsp-1-dim+} and in section~\ref{section-2dim} we prove Theorem~\ref{rsp-vs-wsp}.
Before turning to the proofs, let us introduce some standard notation and recall some basic properties of $C^{k,\alpha}$ functions.

Given $E \subset \R^d$, $f: E \rightarrow \R$, and $\alpha \in (0,1]$, the homogeneous H\"older (semi-) norm is given by %of exponent $\alpha$ of $f$ is
\begin{equation}
    [f]_\alpha := \underset{\substack{x,y \in E \\ x \neq y}}{\sup} \frac{|f(x) - f(y)|}{|x-y|^\alpha}.
\end{equation}

If $E$ is a convex set with non-empty interior and $f$ is of class $C^1$ then the following interpolation inequality holds
(see \cite[Example 1.8]{Lunardi2018} or \cite[Section 3.3]{ABC_2013_LipSard}):
\begin{equation}
\label{useful-inequality}
    [f]_\alpha \leq 2 \|f\|^{1-\alpha}_\infty \|\nabla f\|^\alpha_\infty.
\end{equation}

Given a point $x_0 \in E=[0,1]^d$ and an integer $k \geq 0$,
the H\"older norm is equivalent to the following one:
\begin{equation}
    \|f\|_{k, \alpha} := [\nabla^k f]_\alpha + \sum\limits_{i=0}^{k}|\nabla^i f(x_0)|.
\end{equation}

We conclude this section with the proof of Theorem~\ref{wsp-eq-rsp-for-monotone-fn}:
\begin{proof}[Proof of Theorem~\ref{wsp-eq-rsp-for-monotone-fn}]
By the definition of monotonicity we have $E^*=f^{-1}(P)$,
where the set $P$ is defined by $P:= \{t\in \R \;|\; \H^1(\{f=t\})>0\}$. Therefore $S \setminus E^* \subset f^{-1}(N)$,
where $N:= \{t\in f(\R^2) \;|\; \H^1(\{f=t\})=0\}$.

Let us prove that for any $t\in N$ the set $\{f=t\}$ contains at most one point.
Fix $x\in \{f=t\}$.
It is well-known that a continuum with finite length is arcwise connected
(see e.g. \cite[Lemma 3.13]{Falconer1985}). 
Hence for any $y\in \{f=t\}$ there exists $\gamma \in C([0,1]; \R^2)$
such that $\gamma(0) = x,$ $\gamma(1)= y$ and $\gamma([0,1])\subset \{f=t\}$. 
Then $0=\H^1(\{f=t\}) \ge \H^1(\gamma([0,1]))\ge |x-y|$, and it follows that $y=x$.

Now we claim that $N$ contains at most two points.
Indeed, suppose that $r < s < t$ belong to $N$. Let $x\in \{f=r\}$ and $z\in \{f=t\}$.
Let $\gamma\colon [0,1]\to \R^2$ and $\tilde \gamma \colon[0,1] \to \R^2$
be two disjoint curves connecting $x$ and $z$ (i.e. $\gamma$ and $\tilde \gamma$ are continuous
and $\gamma((0,1)) \cap \tilde \gamma((0,1)) = \emptyset$).
Then by intermediate value property
we can find two distinct points $y \in \gamma((0,1))$ and $\tilde y\in \tilde\gamma((0,1))$
such that $f(y) = f(\tilde y) = s$. This contradicts the property that $\{f=s\}$
contains at most one point.

Since $N$ contains at most two points, and for each $t\in N$ the set $\{f=t\}$ contains at most one point,
the set $f^{-1}(N)$ contains at most two points. Hence $\L^2(S \setminus E^*) \le \L^2(f^{-1}(N)) = 0$.
\end{proof}

\section{One-dimensional case}\label{section-1dim}

We begin with the proof of Theorem~\ref{rsp-1-dim}.

Consider a closed interval $I = [a, b] \subset \R$. Let $F_{I,s} : \R \rightarrow [0, 1]$ be a function such that $F_{I,s} \equiv 1$ on $[a + s, b - s]$ and $F_{I,s} \equiv 0$ on $\R \setminus (a, b)$.

Let $c = \frac{a + b}{2}, \ I_0 = [a, c], \ I_1 = [c, b]$, then
\begin{equation}
H_{I,s}(x) := 0 \cdot F_{I_0,s}(x) + 1 \cdot F_{I_1,s}(x).
\end{equation}

For any $s > 0$ let
\begin{equation}
    P_s(I) := \{[a + s, c - s], [c + s, b - s]\}.
\end{equation}
Similarly, for any finite family $\F$ of disjoint intervals with equal length $l$ for any $s \in (0, l/4)$ let
\begin{equation}
    P_s(\F) := \bigcup_{I \in \F} P_s(I).
\end{equation}
The lengths of the intervals in $P_s(\F)$ are $\frac{1}{2}(l - 4s)$.

Let $\{\alpha_n\}$ be given by
\begin{equation}
\label{alpha-minus}
\alpha_n := \frac{1}{2}n^{-2}.
\end{equation}
Let
\begin{equation}
\label{def-of-r-1dim}
    r := \prod_{n=1}^\infty (1 - \alpha_n).
\end{equation}
Since the series $\sum_{n=1}^\infty \alpha_n$ converge, we have $r>0$.

Let $\{s_n\}$ and $\{a_n\}$ be the sequences such that $a_1 = 1$ and for all $n\in \N$
\begin{equation}
\label{def-of-s-and-a-1dim}
s_n = \frac{1}{4}\alpha_n a_n \qquad
a_{n+1} = \frac{1}{2}(a_n - 4 s_n) = \frac{1}{2}(1-\alpha_n) a_n. 
\end{equation}
By the definition of $r$ we have
\begin{equation}
\label{asim1d}
a_n \sim 2r \cdot \frac{1}{2^n}.
\end{equation}
as $n\to\infty.$

Let $\{\F_n\}$ be the sequence of finite families of disjoint intervals constructed as follows. Let $\F_1 := \{[0, 1]\}$ and for all $n \in \N$ let
\begin{equation}
    \F_{n + 1} := P_{s_n}(\F_n).
\end{equation}
By the definition of $s_n$ and $P_s$ for all $n \in \N$ the family $\F_n$ consists of $2^{n - 1}$ disjoint intervals with length $a_n$.

For any $n \in \N$ let
\begin{equation}
    h_n(x) := \sum_{I \in \F_n} H_{I,s_n}(x).
\end{equation}
Finally, we define
\begin{equation}
\label{def-of-f-1dim}
    f(x) := \sum_{n = 1}^\infty 2^{-n} h_n(x).
\end{equation}
The series converge uniformly on $\R$, since for all $n \in \N$ we have $\|h_n\|_\infty = 1$.

Let $f_n$ denotes partial sums of series \eqref{def-of-f-1dim}.
\begin{figure}[h!]
	\centering
	\includegraphics[width=1\linewidth, trim=0cm 21cm 0cm 3cm]{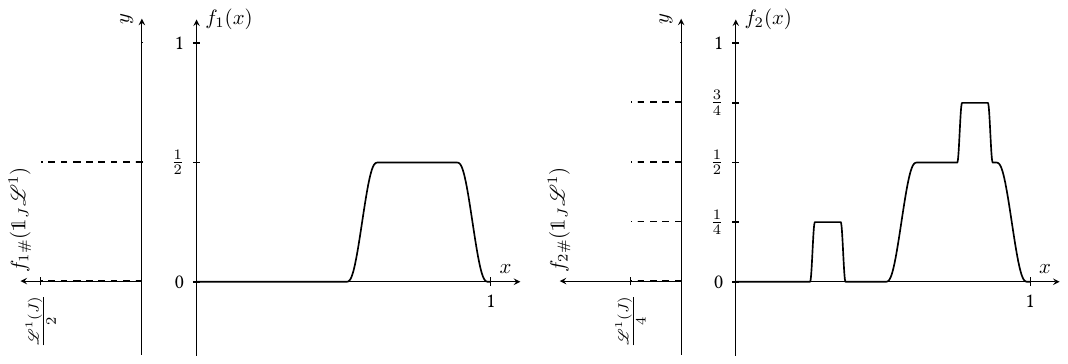}
	\caption{First two steps of function construction and their pushforward measures}
	\label{fig-1-1dim}
\end{figure}

Let $S$ denote the critical set of $f$.
For any $n\in \N$ let $J_n :=\bigcup_{I \in \F_n} I$ denote the union of the $n$-th family of intervals. Let
\begin{equation}
J:=\bigcap_{n\in \N} J_n.
\end{equation}
Then $\L^1(J) = \lim\limits_{n \rightarrow \infty} a_n \cdot 2^{n - 1} = r$.

\begin{prop}
\label{prop-4-1dim}
    We have
    $f \in \bigcap_{\alpha \in (0,1)} C^{0, \alpha} ([0, 1])$ for all $\alpha \in (0, 1)$. Furthermore, $f\in C^\infty([0,1] \setminus J)$.
\end{prop}

\begin{proof}
    Let us estimate $\|f\|_{C^{0, \alpha}} = [f]_{\alpha} + |f(x_0)|$. Since $f$ is bounded, we need only to estimate $[f]_\alpha$. Fix $\alpha \in (0, 1)$, then
    \begin{equation}
        [2^{-n} h_n]_{\alpha} \leq 2 \cdot 2^{-n} \|h_n\|_\infty^{1 - \alpha} \|h_n'\|_\infty^\alpha = O\Big(\frac{1}{2^n s_n^\alpha}\Big)
    \end{equation}
    The series $\sum_{n = 1}^\infty \frac{1}{2^n s_n^\alpha}$ converges
    by the definition of $\{\alpha_n\}$.
    So $f \in C^{0, \alpha} ([0, 1])$ for all $\alpha \in (0, 1)$.

    Outside every $J_n$ the function $f$ agrees with a finite sum of smooth functions, so it is smooth.
\end{proof}

\begin{prop}
\label{prop-5-1dim}
We have $J \subset S$ (modulo $\L^1$).
\end{prop}

\begin{proof}
Let us show that $f$ is not differentiable at a.e. $x \in J$. Consider $x_0 \in J$.

Consider $f(x_0)$ in base 2 system, $(f(x_0))_2 = 0, \! b_1b_2b_3...$. Let $\mathcal{N}(x_0) \subset \N$ be the set of all $n$ such that $b_n = 1$.
Clearly $\mathcal{N}(x)$ is infinite for a.e. $x \in J$.

Let us construct $h_n^l$ and $h_n^r$ for all $n \in \mathcal{N}(x_0)$ as follows.
For $n \in \mathcal{N}(x_0)$ let $I_n \in \F_n$ be such that $P_{s_n}(I_n) = \{I_{n,0}, I_{n,1}\}$, and $x_0 \in I_{n,1}$, then let $I_{n,1} = \left[\lambda, \rho\right]$.
Define $h_n^l$ and $h_n^r$ such that $x_0 + h_n^l = \lambda$ and $x_0 + h_n^r = \rho$.

By construction we have
\begin{equation}
    \Delta f_n:= f(x_0 + h_n^l) - f(x_0) = f(x_0 + h_n^r) - f(x_0) = \sum\limits_{k=1}^{n-1} b_k 2^{-k} -\sum\limits_{k=1}^\infty b_k 2^{-k} = -\sum\limits_{k=n}^\infty b_k 2^{-k},
\end{equation}
\begin{equation}
    |\Delta f_n| \geq 2^{-n}.
\end{equation}

Note that
\begin{equation}
    \big|h_n^{l, r}\big| \leq s_n + a_{n+1} = \frac{1}{4}(2 - \alpha_n) a_n,
\end{equation}
\begin{equation}
    \Bigg|\frac{1}{h_n^{l, r}}\Bigg| \geq \frac{2^n}{r_n},
\end{equation}
where $r_n$ is partial product of \eqref{def-of-r-1dim}.

By the definition of $\{\alpha_n\}$ there is $\delta > 0$ such that for every $n \in \mathcal{N}(x_0)$ we have
\begin{equation}
    \label{finite-difference-estimate}
    \Bigg|\frac{\Delta f_n}{h_n^{l, r}}\Bigg| \geq \frac{1}{r_n} > \delta > 0.
\end{equation}

For all $n \in \mathcal{N}(x_0)$ we have $h_n^r > 0$ and, consequently, $\frac{\Delta f_n}{h_{n}^r} < -\delta$.
Similarly, for all $n \in \mathcal{N}(x_0)$ we have $h_n^l < 0$ and, consequently, $\frac{\Delta f_n}{h_{n}^l} > \delta$. Then $f$ is not differentiable at $x_0$ and $J \subset S$ (modulo $\L^1$).
\end{proof}

\begin{prop}
\label{prop-6-1dim}
We have
$f_\# (\1_J \L^1) = \L^1(J) \1_{[0,1]} \L^1$.
\end{prop}

\begin{proof}
Firstly note that
\begin{equation}
\label{f_sharp-formula-1dim}
(f_n)_\#(\1_J \L^1) = (\L^1(J)) \frac{1}{2^n} \sum_{k = 0}^{2^n - 1} \delta_\frac{k}{2^n}.
\end{equation}
E.g. for $n = 1$ we have $f = \frac{k}{2}$ on $I_k(s_1) \cap J$, $k = \{0,1\}$. Note that $J = \bigcup_{k = 0}^1 (I_k(s_1) \cap J)$. Moreover $\L^1 (I_k(s_1) \cap J) = \frac{1}{2} \L^1 (J)$. Therefore we have 
\begin{equation}
(f_1)_\#(\1_J \L^1) = \sum_{k = 0}^1 \left(\L^1 (I_k(s_1) \cap J) \cdot \delta_\frac{k}{2}\right) = (\L^1 (J)) \frac{1}{2} \sum_{k = 0}^1 \delta_\frac{k}{2}.
\end{equation}
Similarly we get \eqref{f_sharp-formula-1dim} for all $n \in \N$.

Then $\forall \varphi \in C[0, 1]$ we have
\begin{equation}
\int_0^1 \varphi(x) d (f_n)_\#(\1_J \L^1) = (\L^1(J)) \frac{1}{2^n} \sum_{k = 0}^{2^n - 1} \varphi\left(\frac{k}{2^n}\right) \rightarrow (\L^1(J)) \int_0^1 \varphi(x) d \L^1,
\end{equation}
which means that $(f_n)_\#(\1_J \L^1)$ weakly-$*$ converges to $(\L^1(J)) \1_{[0, 1]} \L^1$.

Therefore $f_\#(\1_J \L^1) = (\L^1(J)) \1_{[0, 1]} \L^1$, because $f_n$
converges to $f$ uniformly.
\end{proof}

Note that $f(S \setminus J) \subset \Q$ so $f_\#(\1_S \L^1) \geq f_\#(\1_J \L^1) = \L^1(J)\1_{[0,1]}\L^1$.

Since~$\alpha_n$ is given by \eqref{alpha-minus} then $\L^1(J) = r > 0$, so $f_\#(\1_S \L^1)~=~r\1_{[0,1]}\L^1$ and $f$ does not have the relaxed Sard property.
This completes the proof of Theorem~\ref{rsp-1-dim}.

In order to prove Theorem~\ref{rsp-1-dim+}, one has to repeat the above construction with
\begin{equation}
\label{alpha-plus}
\alpha_n := \frac{1}{n + 1}.
\end{equation}
In this case we have $r=0$, where $r$ is given by \eqref{def-of-r-1dim}.
Then \eqref{asim1d} does not hold, but still $a_n = 2r_n \cdot \frac{1}{2^n}$, where $r_n$ is partial product of \eqref{def-of-r-1dim}, so \eqref{finite-difference-estimate} is satisfied and Proposition~\ref{prop-5-1dim} holds.
The series $\sum_{n = 1}^\infty \frac{1}{2^n s_n^\alpha}$ converges so we have Proposition~\ref{prop-4-1dim}.
Proposition~\ref{prop-6-1dim} holds too. Therefore since $\alpha_n$ is given by \eqref{alpha-plus} then $\L^1(J) = r = 0$, so $f_\#(\1_J \L^1)~=~0$ and $f$ has the relaxed Sard property.
This completes the proof of Theorem~\ref{rsp-1-dim+}.

In the examples constructed above \eqref{strong-Sard-property} does not hold, since $\L^1(f(S \cap J)) \geq \L^1(\{f(x): \#\mathcal{N}(x) = \infty\}) = 1$, where $\mathcal{N}(x)$ was defined in Proposition~\ref{prop-5-1dim}.

\section{Two-dimensional case}\label{section-2dim}

In order to prove Theorem~\ref{rsp-vs-wsp} we will use the similar construction, as in the one-dimensional case.

Let $\psi_{a,b}$ be the function of class $C^\infty(\R)$ such that $\psi_{a,b} \equiv 1$ on $[-b,b]$, $\psi_{a,b} \equiv 0$ outside $(-a,a)$, and $\psi_{a,b}(x) \in [0,1]$ for all $x \in (-a,a) \setminus [-b,b]$.

Consider a closed square $Q \subset \R^2$ with side length $a$ and center $x_0$. For all $x \in \R^2$ let
\begin{equation}
\label{def-of-F}
F_{Q,s}(x) := \psi_{\frac{a}{2},\frac{a}{2}-s}(x_1 - x_{01})\psi_{\frac{a}{2},\frac{a}{2}-s}(x_2 - x_{02}).
\end{equation}

Note that $\|F_{Q,s}\|_\infty = 1$ and $\|\nabla F_{Q,s}\|_{\infty} = \|\psi'_{s,0}\|_{\infty} = \frac{1}{s}\|\psi'_{1,0}\|_{\infty}.$

Let $Q_{0}, Q_{1}, Q_{2}, Q_{3}$ denote the four equal squares with side length $a/2$,
whose union is $Q$ (numbered like the quadrants in the Cartesian coordinate system).
Let
\begin{equation}\label{def-of-H}
H_{Q,s}(x):=\sum_{k=0}^3 k\cdot F_{Q_k, s}(x).
\end{equation}
By the definition $\|H_{Q,s}\|_\infty = 3$ and $\|\nabla H_{Q,s}\|_{\infty} = 3\|\nabla F_{Q_3,s}\|_{\infty} = 3\|\psi'_{s,0}\|_{\infty}  = \frac{3}{s}\|\psi'_{1,0}\|_{\infty}$.

For any $s>0$ let $Q(s)$ denote the smaller square
which is obtained by removing from the square $Q$ the $s$-neighbourhood of its boundary $\d Q$.
Let
\begin{equation*}
P_s(Q) := \{Q_{0}(s), Q_{1}(s), Q_{2}(s), Q_{3}(s)\}.
\end{equation*}
Similarly, for any finite family $\F$ of disjoint squares with equal side length $a>0$ for any $s\in(0, a/4)$ let
\begin{equation*}
P_s(\F):= \bigcup_{Q\in \F} P_s(Q). 
\end{equation*} 
The sides of the squares in $P_s(\F)$ are $\frac{1}{2}(a - 4s)$.

Let $\alpha_n := \frac{1}{2}n^{-2}.$ Then the product
\begin{equation}
r:=\prod_{n=1}^\infty (1-\alpha_n)
\end{equation}
is strictly positive, since the series $\sum_{n=1}^\infty \alpha_n$ converge.

Let $\{s_n\}$ and $\{a_n\}$ be the sequences such that $a_1 = 1$ and for all $n\in \N$
\begin{equation}
\label{def-of-s-and-a}
s_n = \frac{1}{4}\alpha_n a_n \qquad
a_{n+1} = \frac{1}{2}(a_n - 4 s_n) = \frac{1}{2}(1-\alpha_n) a_n. 
\end{equation}
By the definition of $r$ we have
\begin{equation}
a_n \sim 2r \cdot \frac{1}{2^n},
\qquad
s_n \sim \frac{r}{4} \cdot \frac{1}{n^2 2^n}.
\end{equation}
as $n\to\infty.$

Let $\{\F_n\}$ be the sequence of finite families of disjoint squares constructed as follows.
Let $\F_1 := \{[0,1]^2\}$ and for all $n\in \N$ let
\begin{equation}\label{def-of-family-F}
\F_{n+1} := P_{s_n}(\F_n).
\end{equation}
By the definition of $s_n$ (and $P_s$) for all $n\in \N$ the family $\F_n$ consists of $4^{n-1}$ disjoint squares with side length~$a_n$.

For any $n\in \N$ let
\begin{equation}
h_n(x) := \sum_{Q \in \F_n} H_{Q,s_n}(x).
\end{equation}
Finally, we define
\begin{equation}
\label{def-of-f}
f(x) := \sum_{n=1}^\infty 4^{-n} h_n(x).
\end{equation}
The series converge uniformly on $\R^2$, since for all $n\in \N$ we have $\|h_n\|_\infty = 3$.

Let $f_n$ denotes partial sums of series \eqref{def-of-f}. 

\begin{figure}
    \centering
    \includegraphics[width=1\linewidth, trim=0cm 20cm 0cm 3cm]{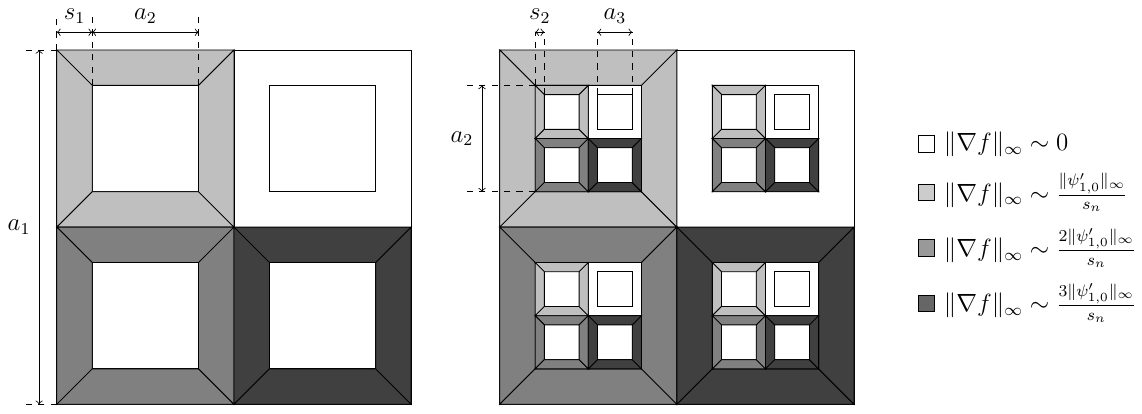}
    \caption{First two steps of function construction}
    \label{fig-1}
\end{figure}

Let $S$ denote the critical set of $f$.
For any $n\in \N$ let $K_n :=\bigcup_{Q \in \F_n} Q$ denote the union of the $n$-th family of squares. Let
\begin{equation}
K:=\bigcap_{n\in \N} K_n.
\end{equation}

\begin{prop}
\label{prop-4}
The function $f$ is $C^{1,\alpha}$ function for all $\alpha \in (0,1)$.
\end{prop}

\begin{proof}
Fix $\alpha \in (0,1)$. Let us estimate the $C^{1,\alpha}$ norm of $f$. Since $f$ and $\nabla f$ are bounded, we need only to estimate $[\nabla f]_\alpha$. Let us calculate the following $L^\infty$-norms:
\begin{equation}
\|\nabla h_n\|_\infty = \|\nabla H_{Q,s_n}\|_\infty = 3\left\|\psi'_{\frac{a_n}{4},\frac{a_n}{4}-s_n}\right\|_\infty = O\left(\frac{1}{s_n}\right) = O(n^2 2^n)
\end{equation}

\begin{equation}
\| \nabla^2  h_n \|_\infty = 3 \cdot \max\left\lbrace \left\|\psi''_{\frac{a_n}{4},\frac{a_n}{4}-s_n}\right\|_\infty, \left\|\psi'_{\frac{a_n}{4},\frac{a_n}{4}-s_n}\right\|_\infty^2 \right\rbrace = O\left(\frac{1}{s_n^2}\right) = O(n^4 2^{2n})
\end{equation}

Using previous calculations and \eqref{useful-inequality} we obtain
\begin{equation}
[\nabla 4^{-n} h_n]_{\alpha} \leq 2 \cdot 4^{-n} \cdot \|\nabla h_n\|_\infty^{1-\alpha} \|\nabla^2 h_n\|_\infty^\alpha = O(2^{-(1-\alpha)n}n^{2+2\alpha})
\end{equation}
Therefore the series \eqref{def-of-f} converge in norm in the Banach space $C^{1,\alpha}$ and $f$ is of class $C^{1,\alpha}$.
\end{proof}

\begin{prop}
\label{prop-5}
We have $K \subset S$.
\end{prop}

\begin{proof}

Consider $(x_0,y_0) \in K$. We claim that $\frac{\d f}{\d x}(x_0,y_0) = 0$.

Consider intersection $\R\times\lbrace y_0\rbrace \bigcap K_n$ for some $n$. There is $Q \in \F_n$ such that $(x_0,y_0)\in Q$. Then $\R\times\lbrace y_0\rbrace \bigcap Q = \{ \lambda_n, \rho_n \} \times \lbrace y_0\rbrace$. By construction $f(\lambda_n, y_0)=f(\rho_n, y_0)$, hence there exists $x_n \in [\lambda_n, \rho_n]$ such that $\frac{\d f}{\d x}(x_n,y_0)=0$. Hence there exists a sequence $\lbrace x_n\rbrace$ such that $\frac{\d f}{\d x}(x_n,y_0)=0$ for all $n$. 
Let $d_n := \diam Q$, where $Q\in \F_n$. By the construction we have $d_n \rightarrow 0$,
hence $x_n \rightarrow x_0$. By the previous proposition, $\frac{\d f}{\d x}$ is continuous, then we have $\frac{\d f}{\d x}(x_0,y_0)=0$. Similarly, $\frac{\d f}{\d y}(x_0,y_0)=0$. Therefore $K \subset S$.
\end{proof}

\begin{prop}
\label{prop-6}
We have $f_\# (\1_K \L^2) = \L^2 (K) \1_{[0,1]} \L^1$.
\end{prop}

\begin{proof}
The proof is similar to the proof of Proposition~\ref{prop-6-1dim}.
\if 0
Firstly note that
\begin{equation}
\label{f_sharp-formula}
(f_n)_\#(\1_K \L^2) = (\L^2(K)) \frac{1}{4^n} \sum_{k = 0}^{4^n - 1} \delta_\frac{k}{4^n}.
\end{equation}
E.g. for $n = 1$ we have $f = \frac{k}{4}$ on $Q_k(s_1) \cap K$, $k = \overline{0,3}$. Note that $K = \bigcup_{k = 0}^3 (Q_k(s_1) \cap K)$. Moreover $\L^2 (Q_k(s_1) \cap K) = \frac{1}{4} \L^2 (K)$. Therefore we have 
\begin{equation}
(f_1)_\#(\1_K \L^2) = \sum_{k = 0}^3 \left(\L^2 (Q_k(s_1) \cap K) \cdot \delta_\frac{k}{4}\right) = (\L^2 (K)) \frac{1}{4} \sum_{k = 0}^3 \delta_\frac{k}{4}.
\end{equation}
Similarly we get \eqref{f_sharp-formula} for all $n \in \N$.

Then $\forall \varphi \in C[0, 1]$ we have
\begin{equation}
\int_0^1 \varphi(x) d (f_n)_\#(\1_K \L^2) = (\L^2(K)) \frac{1}{4^n} \sum_{k = 0}^{4^n - 1} \varphi\left(\frac{k}{4^n}\right) \rightarrow (\L^2(K)) \int_0^1 \varphi(x) d \L^1,
\end{equation}
which means that $(f_n)_\#(\1_K \L^2)$ weakly-$*$ converges to $(\L^2(K)) \1_{[0, 1]} \L^1$. Therefore $f_\#(\1_K \L^2) = (\L^2(K)) \1_{[0, 1]} \L^1$, because $f_n$ converges to $f$ uniformly, and we have $c = \L^2(K) > 0$.
\fi
\end{proof}

\begin{prop}
\label{prop-7}
The function $f$ has the weak Sard property.
\end{prop}

\begin{proof}
For any $x\in [0,1]^2$ let $C_x$ denote the connected component of the level set of $f$ passing through $x$.

The proposition is a consequence of next two statements
\begin{enumerate}
\item[(i)] $S \setminus K \subset f^{-1}(\Q)$ (modulo $\L^2$);
\item[(ii)] $K \cap E^* = \varnothing$ (modulo $\L^2$).
\end{enumerate}

Let us consider $F_{Q,s}$ defined by \eqref{def-of-F}. Before \eqref{def-of-H} we have defined $Q_0$ for a square $Q$. Let $\lbrace Z_n \rbrace$ be the sequence defined as follows. Let $Z_0 := \varnothing$ and 
\begin{equation}
\label{def-of-Z}
    Z_n := \biggl(\bigcup\limits_{Q \in \mathcal{F}_n} Q_0\biggr) \setminus K_{n+1}.
\end{equation}
\begin{figure}[h!]
	\centering
	\includegraphics[width=0.9\linewidth, trim=0cm 21cm 0cm 2cm]{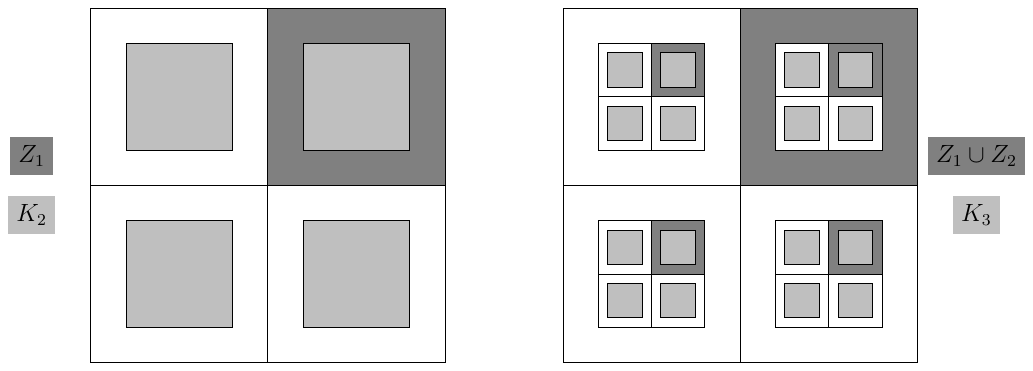}
	\caption{Critical sets of $f_1$ and $f_2$}
	\label{fig-2}
\end{figure}

Intersection of $Q$ and critical set of $F_{Q,s}$ is $Q(s)$ (modulo $\L^2$). Hence critical set of $f_n$ is $K_{n+1} \cup (\bigcup_{i \in \overline{1,n}} Z_i)$ (modulo $\L^2$). Clearly $f_n(K_{n+1} \cup (\bigcup_{i \in \overline{1,n}} Z_i)) = \{0, \frac{1}{4^n}, \dots, \frac{4^n-1}{4^n}\}$. On the other hand 
for all $n \in \N$ and for all $x \in [0,1]^2 \setminus K_n$ we have $f(x) = f_n(x)$. Thus $S \setminus K = \bigcup\limits_{n \in \N} Z_n$ (modulo $\L^2$), and $f\Bigl(\bigcup\limits_{n \in \N} Z_n\Bigr) \subset \Q$.

Let us prove that $K \cap E^* = \varnothing$ (modulo $\L^2)$. Fix $x \in K$ and $y \neq x$. Denote $d = \dist(x,y)$ and $d_n = \diam Q$, where $Q \in \F_n$ where $\F_n$ was defined in \eqref{def-of-family-F}. Since $d_n \rightarrow 0$, we can choose $N$ such that $d_N < d$.
Consider $f(x)$ in base 4 system, $(f(x))_4 = 0, \! b_1b_2...$. If the sequence $b_i$ is zero starting from some $\tilde{N}$, then $f(x) \in \Q$. Otherwise there exists $M \geq N$ such that $b_M \neq 0$, and then we can choose $\tilde{Q} \in \mathcal{F}_M$ such that $x \in \tilde{Q}(s_M)$, and $f(x) \neq f(z)$ for all $z \in \tilde{Q} \setminus \tilde{Q}(s_M)$. Thus $y \not\in C_x$. Therefore $C_x = \{ x\}$, so $\H^1(C_x) = 0$. Then $K \cap E^* = \varnothing$ (modulo $\L^2$).

Thus $f_\# (\1_{S\cap E^*}\L^2) \perp \L^1$.
\end{proof}

\begin{proof}[Proof of Theorem~\ref{rsp-vs-wsp}]
Let $f$ be the function constructed in this section.
Proposition~\ref{prop-6} implies that $f$ does not have the relaxed Sard property.
The rest of the statement follows from Propositions~\ref{prop-4} and~\ref{prop-7}.
\end{proof}

\medskip

{\bf Data availability statement.} {\sl Data sharing not applicable to this article as no datasets were generated or analyzed during the current study.}

\medskip

{\bf Conflict of interest statement}. {\sl The authors declare that they have no conflict of interest.}

\medskip

\section{Acknowledgements}

The work of N.G. and R.D. was supported by the RSF project 24-21-00315.
The authors would like to thank Mikhail V. Korobkov, Paolo Bonicatto and Grigory E. Ivanov for the fruitful discussions.

\bibliographystyle{alpha}
\bibliography{references}
\end{document}